\documentclass[a4paper,10pt,leqno]{amsart}
\title{The Segal Conjecture for Infinite Discrete Groups}
\author{L\"uck, W.}
        \address{Mathematisches Institut der Universit\"at Bonn\\
                Endenicher Allee 60\\
                53115 Bonn, Germany}
         \email{wolfgang.lueck@him.uni-bonn.de}
          \urladdr{http://www.him.uni-bonn.de/lueck}
         \date{January 2019}

\usepackage{hyperref}
\usepackage{color}
\usepackage{pdfsync}
\usepackage{calc}
\usepackage{enumerate,amssymb}
\usepackage[arrow,curve,matrix,tips,2cell]{xy}
  \SelectTips{eu}{10} \UseTips
  \UseAllTwocells

\DeclareMathAlphabet\EuR{U}{eur}{m}{n}
\SetMathAlphabet\EuR{bold}{U}{eur}{b}{n}

\theoremstyle{plain}
\newtheorem{theorem}{Theorem}[section]
\newtheorem{lemma}[theorem]{Lemma}

\newtheorem{problem}[theorem]{Problem}

\theoremstyle{definition}

\newtheorem{definition}[theorem]{Definition}
\newtheorem{example}[theorem]{Example}

\newtheorem{remark}[theorem]{Remark}

{\catcode`@=11\global\let\c@equation=\c@theorem}

\newcommand{\comsquare}[8]                   
{\begin{CD}
#1 @>#2>> #3\\
@V{#4}VV @V{#5}VV\\
#6 @>#7>> #8
\end{CD}
}

\newcommand{\xycomsquare}[8]                   
{\xymatrix
{#1 \ar[r]^{#2} \ar[d]^{#4} &
#3 \ar[d]^{#5}  \\
#6\ar[r]^{#7} &
#8
}
}

\newcommand{\xycomsquareminus}[8]                      
{\xymatrix{#1 \ar[r]^-{#2} \ar[d]^-{#4} &
#3 \ar[d]^-{#5}  \\
#6\ar[r]^-{#7} &
#8
}
}

\newcommand{\IC}{{\mathbb C}}

\newcommand{\II}{{\mathbb I}}

\newcommand{\IQ}{{\mathbb Q}}

\newcommand{\IZ}{{\mathbb Z}}

\newcommand{\bfI}{{\mathbf I}}

\let\sect=\S
\newcommand{\curs}{\EuR}

\newcommand{\FGINJ}{\curs{FGINJ}}
\newcommand{\MODULES}{\curs{MODULES}}
\newcommand{\Or}{\curs{Or}}

\newcommand{\aut}{\operatorname{aut}}

\newcommand{\character}{\operatorname{char}}

\newcommand{\coker}{\operatorname{coker}}

\newcommand{\edge}{\operatorname{edge}}

\newcommand{\id}{\operatorname{id}}

\newcommand{\im}{\operatorname{im}}

\newcommand{\ind}{\operatorname{ind}}

\newcommand{\inverselim}{\operatorname{invlim}}

\newcommand{\pr}{\operatorname{pr}}

\newcommand{\Syl}{\operatorname{Syl}}

\newcommand{\tors}{\operatorname{tors}}

\newcommand{\calf}{\mathcal{F}}

\newcommand{\calh}{\mathcal{H}}

\newcommand{\calk}{\mathcal{K}}

\newcommand{\calp}{\mathcal{P}}

\newcommand{\pt}{\ast}

\newcommand{\higherlim}[3]{{\inverselim}_{#1}^{#2}#3}

\newcommand{\invlim}[2]{\higherlim{#1}{}{#2}}

\newcommand{\EGF}[2]{E_{#2}(#1)}                   

\newcommand{\version}[1]                    
{\begin{center} last edited on #1\\
last compiled on \today\\
name of texfile: \jobname
\end{center}
}

\newcounter{commentcounter}

\begin{document}

\typeout{---------------------------- segc.tex --------------------------}

\typeout{----------------------------  segc.tex  ----------------------------}

\maketitle


\typeout{-----------------------  Abstract  ------------------------}

\begin{abstract}
We formulate and prove a version of the Segal Conjecture for infinite groups.
For finite groups it reduces to the original version.
The condition that $G$ is finite is replaced in our setting
by the assumption that there exists  a finite model for
the  classifying space $\underline{E}G$  for proper actions.
This assumption is satisfied for instance for word hyperbolic groups
or cocompact discrete subgroups of Lie groups
with finitely many path components. As a consequence we get for such groups
$G$ that the zero-th stable cohomotopy of the classifying space $BG$
is isomorphic to the $I$-adic completion of the ring
given by the zero-th equivariant stable cohomotopy of $\underline{E}G$
for $I$ the augmentation ideal.
\end{abstract}


\typeout{--------------------   Section 0: Introduction --------------------------}

\setcounter{section}{-1}
\section{Introduction}
\label{sec:Introduction}

We first recall the Segal Conjecture for a finite group $G$.
The equivariant stable cohomotopy groups $\pi^n_G(X)$ of a $G$-$CW$-complex
are modules over the ring $\pi^0_G = \pi^0_G(\pt)$
which can be identified with the Burnside ring $A(G)$. Here and elsewhere $\pt$ denotes the one-point-space
The augmentation homomorphism
$A(G) \to \IZ$ is the ring homomorphism sending
the class of a finite set to its cardinality.  The
augmentation ideal $\II_G \subseteq A(G)$ is its kernel. Let
$\pi^m_G(X)\widehat{_ {\II_G}}$ be the $\II_G$-adic completion
$ \invlim{n \to \infty}{\pi^m_G(X)/\II_G^n \cdot \pi^m_G(X)}$ of $\pi^m_G(X)$.

The following result was formulated as a conjecture by Segal and proved
by Carlsson~\cite{Carlsson(1984)}.

\begin{theorem}[Segal Conjecture for finite groups]
\label{the:Segal_Conjecture_for_finite_groups}
For every finite group $G$ and
finite $G$-$CW$-complex $X$ there is an isomorphism
\[
\pi^m_G(X)\widehat{_ {\II_G}} \xrightarrow{\cong} \pi^m_s(EG \times_G X).
\]
\end{theorem}

In particular we get for $X = \pt$ and $m = 0$ an isomorphism
\[
A(G)\widehat{_ {\II_G}}  \xrightarrow{\cong}  \pi^0_s(BG).
\]

The purpose of this paper is to formulate and prove a version of it for infinite
(discrete) groups, i.e., we will show

\begin{theorem}{(Segal Conjecture for infinite groups).}
\label{the:Segal_Conjecture_for_infinite_groups} Let $G$ be a
(discrete) group. Let $X$ be a finite proper $G$-$CW$-complex.
Let $L$ be a proper finite dimensional $G$-$CW$-complex with the property that
there is an upper bound on the order of its isotropy groups. 
Let $f \colon X \to L$ be a $G$-map.

Then the map of pro-$\IZ$-modules
\[
\lambda^m_G(X) \colon \left\{\pi^m_G(X)/\II_G(L)^n \cdot \pi^m_G(X)\right\}_{n \ge 1}
\xrightarrow{\cong} \left\{\pi_s^m\left((EG \times_G X)_{(n-1)}\right)\right\}_{n \ge 1}
\]
defined
in~\eqref{map_lambda_of_pro-modules} is an isomorphism of pro-$\IZ$-modules.

In particular we obtain an isomorphism
\[
 \pi^m_G(X)\widehat{_ {\II_G(L)}} \cong \pi_s^m(EG \times_G X).
\]
If there is a
finite $G$-$CW$-model for $\underline{E}G$, we obtain an isomorphism
\[
\pi^m_G(\underline{E}G)\widehat{_ {\II_G(\underline{E}G)}} \cong \pi_s^m(BG).
\]
\end{theorem}

Here $\underline{E}G$ is the classifying space for proper $G$-actions and $\pi^*_G(X)$ is
equivariant stable cohomotopy as defined in~\cite[Section~6]{Lueck(2005r)}.  The ideal
$\II_G(L)$ is the augmentation ideal in the ring $\pi^0_G(L)$ (see Definition~\ref{def:augmentation_ideal}).
We view $\pi^*_G(X)$ as
$\pi^0_G(L)$-module by the multiplicative structure on equivariant stable cohomotopy and
the map $f$. We denote by $\pi^m_G(X)\widehat{_ {\II_G(L)}}$ its $\II_G(L)$-completion.
More explanations will follow in the main body of the text.

In~\cite{Lueck(2005r)} various mutually distinct notions of a Burnside ring of  a group $G$ are 
introduced which all agree with the standard notion for finite  $G$.
If there is a finite $G$-$CW$-model for $\underline{E}G$, then the homotopy theoretic
definition is $A(G) := \pi^0_G(\underline{E}G)$, we define the ideal $\bfI_G \subseteq A(G)$
to be $\bfI_G(\underline{E}G)$,
and we get in this notation from Theorem~\ref{the:Segal_Conjecture_for_infinite_groups} an isomorphism
\[
A(G)\widehat{_{I_G}} \cong \pi_s^0(BG).
\]

We will actually formulate for every equivariant cohomology theory $\calh_?^*$ with multiplicative structure a
Completion Theorem (see Problem~\ref{pro:formulation_of_the_Completion_Theorem}).  It is
not expected to be true in all cases. We give a strategy for its proof in
Theorem~\ref{the:strategy_of_proof_of_the_completion_theorem}.  We show that
this  applies to equivariant stable cohomotopy, thus proving
Theorem~\ref{the:Segal_Conjecture_for_infinite_groups}. It also applies to equivariant
topological $K$-theory, where the Completion Theorem for infinite groups has already been proved
in~\cite{Lueck-Oliver(2001a)}.

If $G$ is finite, we can take $L = \underline{E}G = \pt$ and then
Theorem~\ref{the:Segal_Conjecture_for_infinite_groups} reduces to
Theorem~\ref{the:Segal_Conjecture_for_finite_groups}. We will not give a new proof of
Theorem~\ref{the:Segal_Conjecture_for_finite_groups}, but use it as input in the proof of
Theorem~\ref{the:Segal_Conjecture_for_infinite_groups}.

This paper is part of a general program to systematically study equivariant homotopy theory,
which is well-established for finite groups and compact Lie groups, for
infinite groups and non-compact Lie groups. The motivation comes among other things
from the Baum-Connes Conjecture and the Farrell-Jones Conjecture.

\subsection*{Acknowledgement}
The paper has been  financially supported by the Leibniz-Preis of the author
 granted by the {DFG}, the ERC Advanced Grant ``KL2MG-interactions''
 (no.  662400) of  the author granted by the European Research Council, and
 by the Deutsche Forschungsgemeinschaft (DFG, German Research Foundation)
under Germany's Excellence Strategy \--- GZ 2047/1, Projekt-ID 390685813.by the Cluster of Excellence 
``Hausdorff Center for Mathematics'' at Bonn. The author wants to thank the referee for his detailed and useful report.


\typeout{--------------------   Section 1 ---------------------------------------}

\section{Equivariant Cohomology Theories with Multiplicative Structure}
\label{sec:Equivariant_Cohomology_Theories_with_Multiplicative_Structures}

We briefly recall the axioms of a (proper) equivariant cohomology $\calh_?^*$ theory with
values in $R$-modules with multiplicative structure.  More details can be found
in~\cite{Lueck(2005c)}.

Let $G$ be a (discrete) group. Let $R$ be a commutative ring with
unit. A \emph{(proper) $G$-cohomology theory $\calh^*_G$ with values in $R$-modules} assigns to any pair
$(X,A)$ of (proper) $G$-$CW$-complexes $(X,A)$ a $\IZ$-graded $R$-module $\{\calh^n_G(X,A)
\mid n \in \IZ\}$ such that $G$-homotopy invariance holds and there exists long exact
sequences of pairs and long exact Mayer-Vietoris sequences. Often one also requires the disjoint union axiom
what we will need not here since all our disjoint unions will be over finite index sets.

A \emph{multiplicative structure} is given by a collection of $R$-bilinear pairings
\begin{eqnarray*}
\cup \colon \calh^{m}_G(X,A) \otimes_R \calh^{n}_G(X,B) & \to &
\calh^{m+n}_G(X,A\cup B).
\end{eqnarray*}
This product is required to be graded commutative, to be associative, to have a unit $1
\in \calh^0_G(X)$ for every (proper) $G$-$CW$-complex $X$, to be compatible with
boundary homomorphisms and to be natural with respect to $G$-maps.

Let $\alpha\colon H \to G$ be a group homomorphism.
Given an $H$-space $X$, define the \emph{induction of $X$ with $\alpha$}
to be the $G$-space $\ind_{\alpha} X$ which  is the quotient of
$G \times X$ by the right $H$-action
$(g,x) \cdot h := (g\alpha(h),h^{-1} x)$
for $h \in H$ and $(g,x) \in G \times X$.
If $\alpha\colon H \to G$ is an inclusion, we also write $\ind_H^G$ instead of
$\ind_{\alpha}$.

A \emph{(proper) equivariant cohomology theory
$\calh_?^*$ with values in $R$-modules}
consists of a collection of (proper)
$G$-cohomology theories $\calh^*_G$ with values in $R$-modules for each group $G$
together with the following so called \emph{induction structure}:
given a group homomorphism $\alpha\colon  H \to G$ and  a (proper) $H$-$CW$-pair
$(X,A)$ there are for each $n \in \IZ$ natural homomorphisms
\begin{eqnarray}
\ind_{\alpha}\colon \calh^n_G(\ind_{\alpha}(X,A))
&\to &
\calh^n_H(X,A)\label{induction_structure}
\end{eqnarray}
If $\ker(\alpha)$ acts freely on $X$, then
$\ind_{\alpha}\colon \calh^n_G(\ind_{\alpha}(X,A))
\to \calh^n_H(X,A)$
is bijective for all $n \in \IZ$.
The induction structure is required to be compatibility with the boundary homomorphisms,
to be functorial in $\alpha$ and to be compatible with inner automorphisms.

Sometimes we will need the following lemma whose elementary proof is analogous to the
one in~\cite[Lemma 1.2]{Lueck(2002b)}.

 \begin{lemma}\label{lem:calh_G(G/H)_and_calh_H(ast)}
 Consider finite subgroups $H,K \subseteq G$ and an element
 $g \in G$ with $gHg^{-1} \subseteq K$.
 Let $R_{g^{-1}}\colon G/H \to G/K$ be the $G$-map
 sending $g'H$ to $g'g^{-1}K$ and
 $c(g)\colon H \to K$ be the group homomorphism sending $h$ to $ghg^{-1}$.
 Denote by $\pr\colon (\ind_{c(g)\colon H \to K}\pt) \to \pt$ the projection
 to the one-point space $\pt$.
 
 Then the following diagram commutes
 \[
\xymatrix@!C=9em{
 \calh_G^n(G/K)  \ar[r]^{\calh^n_G(R_{g^{-1}})} \ar[d]_{\ind_K^G}
&
\calh_G^n(G/H) \ar[d]^{\ind_H^G}
  \\
 \calh^n_K(\pt) \ar[r]_{\ind_{c(g)} \circ \calh_K^n(\pr)}
&
\calh^n_H(\pt)
}
\]
 \end{lemma}

 Let $\calh_?^*$ be a (proper) equivariant cohomology theory. A \emph{multiplicative
   structure} on it assigns a multiplicative structure to the associated (proper)
 $G$-coho\-mo\-logy theory $\calh^*_G$ for every group $G$ such that for each group
 homomorphism $\alpha \colon H \to G$ the maps given by the induction structure
 of~\eqref{induction_structure} are compatible with the multiplicative structures on
 $\calh^*_G$ and $\calh^*_H$.

\begin{example}{\bf Equivariant cohomology theories coming from non-equi\-va\-riant ones.}%
\label{exa:equivariant_cohomology_theories}
 Let $\calk^*$ be a (non-equivariant) cohomology theory with multiplicative structure,
 for instance singular cohomology or topological $K$-theory. We can assign to it an
 equivariant cohomology theory with multiplicative structure $\calh_?^*$ in two ways.
 Namely, for a group $G$ and a pair of $G$-$CW$-complexes $(X,A)$ we define
 $\calh^n_G(X,A)$ by $\calk^n(G\backslash (X,A))$ or by $\calk^n(EG \times_G(X,A))$. 
\end{example}

\begin{example}{\bf (Equivariant topological $K$-theory).}
\label{exa:equivariant_topological_K-theory}
In~\cite{Lueck-Oliver(2001a)} equivariant topological $K$-theory is defined for finite
proper equivariant $CW$-complexes in terms of equivariant vector bundles. 
It reduces to the classical notion which appears for instance in~\cite{Atiyah(1967)}. Its
relation to equivariant $KK$-theory is explained in~\cite{Phillips(1988)}.  This
definition is extended to (not necessarily finite) proper equivariant $CW$-complexes
in~\cite{Lueck-Oliver(2001a)} in terms of equivariant spectra using $\Gamma$-spaces and yields a proper
equivariant cohomology theory $K_?^*$ with multiplicative as explained in~\cite[Example~1.7]{Lueck(2005c)}.
It has the property that for any finite subgroup $H$ of
a group $G$ we have
\[
K^n_G(G/H)  =  K^n_H(\pt)  =  \left\{
\begin{array}{lll}
R_{\IC}(H) & & n \text{ even};
\\
\{0\}       & & n \text{ odd},
\end{array}
\right.
\]
where $R_{\IC}(H)$ denote the complex representation ring of $H$.
\end{example}

\begin{example}{\bf (Equivariant Stable Cohomotopy).}
\label{exa:equivariant_cohomotopy}
In~\cite[Section~6]{Lueck(2005r)} equivariant stable cohomotopy $\pi_?^*$ is defined
for finite proper equivariant $CW$-complexes in terms of maps of sphere bundles associated
of equivariant vector bundles. For finite $G$ it reduces to the classical notion.  This
definition is extended to arbitrary proper $G$-$CW$-complexes 
by Degrijse-Hausmann-Lueck-Patchkoria-Schwede~\cite{Degrijse-Hausmann-Lueck-Patchkoria-Schwede(2019)}, where 
a systematic study of equivariant homotopy theory for (not necessarily compact) Lie groups and proper $G$-$CW$-complexes
is developed.

Let $H \subseteq G$ be a finite subgroup.  Recall that by the induction structure we have
$\pi^n_G(G/H) = \pi^n_H(\pt)$.  The equivariant stable homotopy groups $\pi_H^n$ are
computed in terms of the splitting due to Segal and tom Dieck
(see~\cite[Proposition~2]{Segal(1971)} and~\cite[Theorem~7.7 in Chapter~II on
page~154]{Dieck(1987)}) by
\[
\pi^H_n = \pi_{-n}^H = \bigoplus_{(K)} \pi^s_{-n}(BW_HK),
\]
where $\pi_{-n}^s$
denotes (non-equivariant) stable homotopy and $(K)$ runs through the conjugacy classes of
subgroups of $H$. In particular we get
\[
\begin{array}{lllll} |\pi^n_G(G/H)| & < & \infty & & n \le -1;
  \\
  \pi^0_G(G/H) & = & A(H); & &
  \\
  \pi^n_G(G/H) & = & \{0\} & & n \ge 1,
\end{array}
\]
where $A(H)$ is the Burnside ring. 
\end{example}


\typeout{--------------------   Section 2 ---------------------------------------}

\section{Some Preliminaries about Pro-Modules}
\label{sec:Some_Preliminaries_about_Pro-Modules}

It will be crucial to handle pro-systems and pro-isomorphisms and not
to pass directly to inverse limits.
In this section we fix our notation for handling pro-$R$-modules for a
commutative ring with unit $R$. 
For the definitions in full generality see for instance~\cite[Appendix]{Artin-Mazur(1969)}
or~\cite[\sect 2]{Atiyah-Segal(1969)}.

For simplicity, all pro-$R$-modules dealt
with here will be indexed by the positive  integers.  We
write $\{M_n,\alpha_n\}$ or briefly $\{M_n\}$ for the inverse system
\[
 M_0 \xleftarrow{\alpha_1} M_1 \xleftarrow{\alpha_2} M_2 \xleftarrow{\alpha_3}
M_3 \xleftarrow{\alpha_4} \ldots.
\]
and also write
$\alpha_n^m := \alpha_{m+1} \circ \cdots \circ \alpha_{n}\colon M_n \to M_m$
for $n > m$ and put $\alpha^n_n =\id_{M_n}$.  For the purposes here, it will
suffice (and greatly simplify the notation)
to work with ``strict'' pro-homomorphisms
$\{f_n\} \colon \{M_n,\alpha_n\} \to \{N_n,\beta_n\}$, i.e.,
a collection of  homomorphisms $f_n \colon M_n \to N_n$
for $n \ge 1$ such that  $\beta_{n}\circ f_n = f_{n-1}\circ\alpha_{n}$ holds
for each $ n\ge 2$.  Kernels  and cokernels of strict homomorphisms are
defined in the obvious way, namely levelwise.

A pro-$R$-module $\{M_n,\alpha_n\}$  will be called
\emph{pro-trivial} if for each $m \ge 1$, there
is some $n\ge m$ such that  $\alpha^m_n = 0$.  A strict homomorphism
$f\colon \{M_n,\alpha_n\} \to \{N_n,\beta_n\}$ is a
\emph{pro-isomorphism}
if and only if $\ker(f)$ and $\coker(f)$ are both
pro-trivial, or, equivalently, for each $m\ge 1$ there is some
$n\ge m$ such that $\im(\beta_n^m) \subseteq \im(f_m)$
and $\ker(f_n) \subseteq \ker(\alpha_n^m)$.
A sequence of strict homomorphisms
\[
\{M_n,\alpha_n\} \xrightarrow{\{f_n\}} \{M_n',\alpha_n'\}
\xrightarrow{g_n} \{M_n'',\alpha_n''\}
\]
will be called \emph{exact} if the sequences of $R$-modules
$M_n \xrightarrow{f_n} N_n \xrightarrow{g_n} M_n''$ is
exact for each $n \ge 1$, and it is called \emph{pro-exact} if
$g_n \circ f_n = 0$  holds for $n \ge 1$ and
the pro-$R$-module  $\{\ker(g_n)/\im(f_n)\bigr\}$ is pro-trivial.

The elementary proofs of the following two lemmas can be found for instance
in~\cite[Section~2]{Lueck(2007)}.

\begin{lemma}\label{lem:pro-exactness_and_limits}
Let $0 \to \{M_n',\alpha_n'\} \xrightarrow{\{f_n\}} \{M_n,\alpha_n\}
\xrightarrow{\{g_n\}} \{M_n'',\alpha_n''\} \to 0$ be a pro-exact sequence
of
pro-$R$-modules. Then there is a natural exact sequence
\begin{multline*}
0 \to \invlim{n \ge 1}{M_n'} \xrightarrow{\invlim{n \ge 1}{f_n}}
\invlim{n \ge 1}{M_n} \xrightarrow{\invlim{n \ge 1}{g_n}}
\invlim{n \ge 1}{M_n''} \xrightarrow{\delta}
\\
 \higherlim{n \ge 1}{1}{M_n'} \xrightarrow{\higherlim{n \ge 1}{1}{f_n}}
\higherlim{n \ge 1}{1}{M_n} \xrightarrow{\higherlim{n \ge 1}{1}{g_n}}
\higherlim{n \ge 1}{1}{M_n''} \to 0.
\end{multline*}
In particular a pro-isomorphism
$\{f_n\} \colon \{M_n,\alpha_n\} \to \{N_n,\beta_n\}$ induces isomorphisms
\[
\begin{array}{llcl}
\invlim{n \ge 1}{f_n} \colon & \invlim{n \ge 1}{M_n}
& \xrightarrow{\cong} & \invlim{n \ge 1}{N_n};
\\
\higherlim{n \ge 1}{1}{f_n} \colon & \higherlim{n \ge 1}{1}{M_n}
& \xrightarrow{\cong} & \higherlim{n \ge 1}{1}{N_n}.
\end{array}
\]
\end{lemma}

\begin{lemma}\label{lem:pro-exactness_and_exactness}
Fix any commutative Noetherian ring $R$, and any ideal
$I\subseteq R $.  Then for any exact sequence $M' \to M \to M''$ of
finitely generated $R$-modules, the sequence
\[
 \{M'/I^nM'\} \to  \{M/I^nM\} \to  \{M''/I^nM''\} 
\]
of pro-$R$-modules is pro-exact.
\end{lemma}


\typeout{--------------------   Section 3---------------------------------------}

\section{The Formulation of a Completion Theorem}
\label{sec:The_Formulation_of_a_Completion_Theorem}

Consider a proper equivariant $G$-cohomology theory $\calh_?^*$ with
multiplicative structure.  In the sequel $\calh^*$ is the
non-equivariant cohomology theory with multiplicative structure given
by $\calh^*_G$ for $G = \{1\}$. Notice that $\calh^0(\pt)$ is a
commutative ring with unit and $\calh^n_G(X)$ is a
$\calh^0(\pt)$-module. In some future applications $\calh^0(\pt)$ will be just 
$\IZ$.  In the sequel $[Y,X]^G$ denotes the set of $G$-homotopy
classes of $G$-maps $Y \to X$. Notice that evaluation at the unit
element of $G$ induces a bijection $[G,X]^G \xrightarrow{\cong}
\pi_0(X)$. It is compatible with the left $G$-actions, which are on the source
induced by precomposing with right multiplication $r_g \colon G \to G, g' \mapsto g'g$
and on the target by the given left $G$-action on $X$.

So we can represent elements in $G\backslash \pi_0(X)$ by classes
$\overline{x}$ of $G$-maps $x \colon G \to X$, where
$x \colon G \to X$ and $y \colon G \to X$ are equivalent,
if for some $g \in G$ the composite $y \circ r_g$
is $G$-homotopic to $x$.

\begin{definition}[Augmentation ideal]\label{def:augmentation_ideal}
  For any proper $G$-CW-complex $X$ the \emph{augmentation module} $\II_G^n(X)\subseteq
  \calh^n_G(X)$ is defined as the kernel of the map
  \[
\calh^n_G(X) \xrightarrow{\prod_{\overline{x} \in G\backslash \pi_0(X)} \ind_{\{1\}
      \to G} \circ \calh^n_G(x)} \prod_{\overline{x} \in G\backslash \pi_0(X)}
  \calh^n(\pt).
\]
  (The composite above is independent of the choice of $x \in
  \overline{x}$ by $G$-homotopy invariance and Lemma~\ref{lem:calh_G(G/H)_and_calh_H(ast)}.)
  If $n = 0$, the map above is a ring homomorphism and $\II_G(X) := \II^0_G(X)$ is an
  ideal called \emph{the augmentation ideal}.
\end{definition}

Given a $G$-map $f\colon X \to Y$, the induced map $\calh^n_G(f)
\colon \calh^n_G(Y) \to \calh^n_G(X)$ restricts to a map $\II_G^n(Y) \to \II_G^n(X)$.

We will need the following elementary lemma:

\begin{lemma}\label{lem:nilpotence_of_In}
  Let $X$ be a CW-complex of dimension $(n-1)$.  Then any $n$-fold
  product of elements in $\II_G^*(X)$ is zero.
\end{lemma}
\begin{proof}Write $X=Y\cup A$, where $Y$ and $A$ are closed subsets, $Y$
contains $X^{(n-2)}$ as a homotopy deformation retract, and $A$ is a disjoint
union of $(n{-}1)$-disks. Fix elements $v_1,v_2,\dots,v_n\in
\II_G^*(X)$.  We can assume by induction that $v_1\cdots{}v_{n-1}$
vanishes after restricting to $Y$, and hence that it is the image of
an element $u\in \calh_G^*(X,Y)$.  Also, $v_n$ clearly vanishes after
restricting to $A$, and hence is the image of an element $v\in
\calh_G^*(X,A)$.  The product of $v_1\cdots{}v_{n-1}$ and $v_n$ is the image in $\calh_G^*(X)$ of
the element $uv\in \calh_G^*(X,Y\cup A) = \calh_G^*(X,X) = 0$ and so $v_1 \cdots{} v_n=0$.  
\end{proof}

Now fix a map $f \colon X \to L$ between $G$-$CW$-complexes. Consider
$\calh^*_G(X)$ as a module over the ring $\calh^0_G(L)$.  Consider the composition
\begin{multline*}
  \II_G(L)^n \cdot \calh_G^m(X) \xrightarrow{i} \calh_G^m(X)
  \xrightarrow{\calh_G^m(\pr)} \calh_G^m(EG \times X) \\
  \xrightarrow{\left(\ind_{G \to \{1\}}\right)^{-1}} \calh^m(EG
  \times_G X) \xrightarrow{\calh^m(j)} \calh^m\left((EG \times_G
    X)_{(n-1)}\right),
\end{multline*}
where $i$ and $j$ denote the inclusions, $\pr$ the projection and $(EG
\times_G X)_{(n-1)}$ is the $(n-1)$-skeleton of $EG \times_G X$.  This
composite is zero because of Lemma~\ref{lem:nilpotence_of_In} since
its image is contained in
$\II^n\left((EG \times_G X)_{(n-1)}\right)$. Thus we obtain a homomorphism of
pro-$\calh^0(\pt)$-modules
\begin{multline}
\lambda^m_G(f \colon X \to L) \colon \left\{\calh^m_G(X)/\II_G(L)^n \cdot \calh^m_G(X)\right\}_{n \ge 1}
\\ \to \left\{\calh^m_G\left((EG \times_G X)_{(n-1)}\right)\right\}_{n \ge 1}.
\label{map_lambda_of_pro-modules}
\end{multline}
We will sometimes write $\lambda^m_G$ or $\lambda^m_G(X)$ instead of
$\lambda^m_G(f \colon X \to L)$ if the map $f \colon X \to L$ is clear
from the context. Notice that the target of $\lambda^m_G(f \colon X \to
L)$ depends only on $X$ but not on the map $f \colon X \to L$, whereas
the source does depend on $f$.

\begin{problem}[Completion Problem]\label{pro:formulation_of_the_Completion_Theorem}
  Under which conditions on ${\calh^*_?}$ and $L$ is the map of
  pro-$\calh^0(\pt)$-modules $\lambda^m_G(f \colon X \to L)$ defined
  in~\eqref{map_lambda_of_pro-modules} an isomorphism of
  pro-$\calh^0(\pt)$-modules?
\end{problem}

\begin{remark}[Consequences of the Completion Theorem]\label{rem:Consequences_of_the_Completion_Theorem}
Suppose that the map of pro-$\calh^0(\pt)$-modules
$\lambda^m_G(X)$ defined in~\eqref{map_lambda_of_pro-modules} is an
isomorphism of pro-$\calh^0(\pt)$-modules.  Obviously the
pro-module $\{\calh^m_G(X)/\II_G(L)^n \cdot \calh^m_G(X)\}_{n \ge 1}$
satisfies the Mittag-Leffler condition since all structure maps are
surjective. This implies that its ${\lim}^1$-term vanishes.  We
conclude from Lemma~\ref{lem:pro-exactness_and_limits}
\begin{eqnarray*}
\higherlim{n \to \infty}{1}{\calh^m\left((EG \times_G X)_{(n-1)}\right)} & = & 0;
\\
\invlim{n \to \infty}{\calh^m\left((EG \times_G X)_{(n-1)}\right)} & \cong &
\invlim{n \to \infty}{\calh^m_G(X)/\II_G(L)^n \cdot \calh^m_G(X)}.
\end{eqnarray*}
Milnor's exact sequence
\begin{multline*}
  0 \to \higherlim{n \to \infty}{1}{\calh^{m-1}\left((EG \times_G
      X)_{(n-1)}\right)} \to \calh^m(EG \times_G X)
  \\
  \to \invlim{n \to \infty}{\calh^m\left((EG \times_G
      X)_{(n-1)}\right)} \to 0
\end{multline*}
implies that we obtain an isomorphism
\[
\calh^m(EG \times_G X) \cong \invlim{n \to
  \infty}{\calh^m_G(X)/\II_G(L)^n \cdot \calh^m_G(X)}.
\]
\end{remark}

\begin{remark}[Taking $L = \underline{E}G$]
\label{rem:taking_L_to_be_underlineEG}
The \emph{classifying space $\underline{E}G$ for proper $G$-actions}
is a proper $G$-$CW$-complex such that the $H$-fixed point set is
contractible for every finite subgroup $H \subseteq G$. It has the
universal property that for every proper $G$-$CW$-complex $X$ there is
up to $G$-homotopy precisely one $G$-map $f \colon X \to
\underline{E}G$. Recall that a $G$-$CW$-complex is proper if and only
if all its isotropy groups are finite and is finite if and only if it
is cocompact.  There is a cocompact $G$-$CW$-model for the classifying
space $\underline{E}G$ for proper $G$-actions for instance if $G$ is
word-hyperbolic in the sense of Gromov, if $G$ is a cocompact subgroup
of a Lie group with finitely many path components, if $G$ is a
finitely generated one-relator group, if $G$ is an arithmetic group, a
mapping class group of a compact surface or the group of outer
automorphisms of a finitely generated free group. For more information
about $\underline{E}G$ we refer for instance to~\cite{Baum-Connes-Higson(1994)}
and~\cite{Lueck(2005s)}.

Suppose that there is a finite model for the classifying space of
proper $G$-actions $\underline{E}G$. Then we can apply this to $\id
\colon \underline{E}G \to \underline{E}G$ and obtain an isomorphism
\[
\calh^m(BG) \cong \invlim{n \to
  \infty}{\calh^m_G(\underline{E}G)/\II_G(\underline{E}G)^n \cdot \calh^m_G(\underline{E}G)}.
\]
\end{remark}

\begin{remark}[The free case]\label{rem:The_torsionfree_case} 
The statement of the Completion Theorem as stated in
Problem~\ref{pro:formulation_of_the_Completion_Theorem}
is always true for trivial reasons if $X$ is a free
finite $G$-$CW$-complex. Then induction induces an isomorphism
\[
\ind_{G \to \{1\}} \colon
\calh^m(G\backslash X) \xrightarrow{\cong} \calh^m_G(X).
\]
Since $\II(G\backslash X)^n = 0$ for large enough $n$ by Lemma~\ref{lem:nilpotence_of_In},
the canonical map
\[
\{\calh^m(G\backslash X)\}_{n \ge 1} \xrightarrow{\cong}
\{\calh^m(G\backslash X)/\II_G(L)^n \cdot \calh^m(G\backslash X)\}_{n \ge 1}
\]
with the constant pro-$\calh^0(\pt)$-module as source is an isomorphism.
Hence the source of $\lambda^m_G(f \colon G \to X)$ can be identified with
constant pro-$\calh^0(\pt)$-module $\{\calh^m(G\backslash X)\}_{n \ge 1}$.

The projection $EG \times_G X \to G\backslash X$ is a
homotopy equivalence and induces an isomorphism
pro-$\IZ$-modules
\[
\left\{\calh^m\left((G\backslash X)_{(n-1)}\right)\right\}_{n \ge 1}  \xrightarrow{\cong}
\left\{\calh^m\left(\left(EG \times_G X\right)_{(n-1)}\right)\right\}_{n \ge 1}.
\]
Since $G\backslash X$ is finite dimensional, the canonical map
\[
\{\calh^m(G\backslash X)\}_{n \ge 1}
 \xrightarrow{\cong}
\left\{\calh^m\left((G\backslash X)_{(n-1)}\right)\right\}_{n \ge 1}
\]
is an isomorphism of pro-$\IZ$-modules. Hence also the target
of $\lambda^m_G(f \colon G \to X)$ can be identified with
constant pro-$\calh^0(\pt)$-module $\{\calh^m(G\backslash X)\}_{n \ge 1}$.
One easily checks that under these identifications
$\lambda^m_G(f \colon G \to X)$ is the identity.

Hence the Completion Theorem is only interesting in the case, where $G$ contains torsion.
\end{remark}


\typeout{--------------------   Section 4  ---------------------------------------}

\section{A Strategy for a Proof of a Completion Theorem}
\label{sec:A_Strategy_for_a_Proof_of_a_Completion_Theorem}

\begin{theorem}{\bf (Strategy for the proof of Theorem~\ref{the:Segal_Conjecture_for_infinite_groups}).}%
\label{the:strategy_of_proof_of_the_completion_theorem}
 Let $\calh^?_*$ be an equivariant cohomology theory with values in $R$-modules with a
 multiplicative structure.  Let $L$ be a proper $G$-$CW$-complex. Suppose that
 the following conditions are satisfied, where $\calf(L)$ is the family of subgroups
 of $G$ given by $\{H \subseteq G \mid L^H \not= \emptyset\}$.
\begin{enumerate}

\item\label{the:strategy_of_proof_of_the_completion_theorem:Notherian}
  The ring $\calh^0(\pt)$ is Noetherian;

\item\label{the:strategy_of_proof_of_the_completion_theorem:fin._gen.}
  For any $H \in \calf(L)$ and $m \in \IZ$ the
  $\calh^0(\pt)$-module $\calh^m_H(\pt)$ is finitely generated;

\item\label{the:strategy_of_proof_of_the_completion_theorem:ideals}
  Let $H \in \calf(L)$, let $\calp \subseteq
  \calh^0_H(\pt)$ be a prime ideal, and let $f \colon G/H \to L$ be any
  $G$-map. Then the augmentation ideal
  \[
\II(H) = \ker \left(\calh^0_H(\pt) \xrightarrow{\calh^0_H(\pr)} \calh^0_H(H)
    \xrightarrow{\ind_{\{1\} \to H}} \calh^0(\pt)\right)
\]
  is contained in $\calp$
  if $\calh^0_G(L) \xrightarrow{\calh^0_G(f)} \calh^0_G(G/H) \xrightarrow{\ind_{G \to \{1\}}}
  \calh^0_H(\pt)$ maps $\II_G(L)$ into $\calp$;

\item\label{the:strategy_of_proof_of_the_completion_theorem:finite_groups}
  The Completion Theorem is true for every finite group $H$ with $H \in \calf(L)$
  in the case, where $X = L = \pt$ and $f = \id \colon \pt \to \pt$. In other words, for every
  finite group $H$ with $L^H \not= \emptyset$ the map of pro-$\calh^0(\pt)$-modules
\begin{eqnarray*}
 \lambda^m_H(\pt) \colon \left\{\calh^m_H(\pt)/\II(H)^n\right\}_{n \ge 1} & \to &
 \left\{\calh^m\left((BH)_{(n-1)}\right)\right\}_{n \ge 1}
\end{eqnarray*}
defined in~\eqref{map_lambda_of_pro-modules} is an isomorphism of
pro-$\calh^0(\pt)$-modules.

\end{enumerate}

Then, under the conditions above,  the Completion Theorem is true for $\calh^?_*$ and every $G$-map
$f \colon X \to L$ from a finite proper $G$-$CW$-complex $X$ to
$L$, i.e.,  the map of
pro-$\calh^0(\pt)$-modules
\begin{eqnarray*}
 \lambda^m_G(X) \colon \left\{\calh^m_G(X)/\II_G(L)^n \cdot \calh^m_G(X)\right\}_{n \ge 1} & \to &
\left\{\calh^m_G\left((EG \times_G X)_{(n-1)}\right)\right\}_{n \ge 1}
\end{eqnarray*}
defined in~\eqref{map_lambda_of_pro-modules} is an isomorphism of
pro-$\calh^0(\pt)$-modules.
\end{theorem}

\begin{proof}
  We first prove the Completion Theorem for $X = G/H$, i.e., for any a $G$-map
  $f \colon G/H \to L$. Obviously $H$ belongs to $\calf(L)$. The following diagram of pro-modules commutes
  \[
\xymatrix@!C=19em{\{\calh^m_G(G/H)/\II_G(L)^n \cdot \calh^m_G(G/H)\}_{n \ge 1}
    \ar[d]_{\{\ind_{H \to G}\}_{n \ge 1}} \ar[r]^-{\lambda^m_G(f \colon
      G/H \to L)} & \{\calh^m\left((EG \times_G G/H)_{(n-1)}\right)\}_{n \ge 1}
    \\
    \{\calh^m_H(\pt)/\II_G(L)^n \cdot \calh^m_H(\pt) \ar[d]_{\pr}\}_{n \ge 1} &
    \\
    \{\calh^m_H(\pt)/\II(H)^n \cdot \calh^m_H(\pt)\}_{n \ge 1}
    \ar[r]_-{\lambda^m_H(\id \colon \pt \to \pt)} &
    \{\calh^m\left((BH)_{(n-1)}\right)\}_{n \ge 1} \ar[uu]_{\{\calh^m_G(\pr)\}_{n \ge 1}} }
 \]
  where $\pr$ denotes the obvious projection.  The lower horizontal
  arrow is an isomorphism of pro-modules by
  condition~\eqref{the:strategy_of_proof_of_the_completion_theorem:finite_groups}.  The
  right vertical arrow and the upper left vertical arrow are
  obviously isomorphisms of pro-modules. 
  Hence the upper horizontal arrow is an isomorphism of pro-modules if
  we can show that the lower left vertical arrow is an isomorphism of
  pro-modules.

  Let $I_f$ be the image of $\II_G(L)$ under the composite of ring homomorphisms
  \[
\calh^0_G(L) \xrightarrow{\calh^0_G(f)} \calh^0_G(G/H)
  \xrightarrow{\ind_{H \to G}} \calh^0_H(\pt)
\]
  Let $J_f$ be the ideal
  in $\calh^0_H(\pt)$ generated by $I_f$.  Obviously $I_f \subseteq J_f \subseteq \II(H)$.
  Then the left lower vertical arrow is the composite
  \begin{multline*}
\calh^m_H(\pt)/\II_G(L)^n \cdot \calh^m_H(\pt) \to
  \calh^m_H(\pt)/(J_f)^n \cdot \calh^m_H(\pt) 
  \\
   \to
  \calh^m_H(\pt)/\II(H)^n \cdot \calh^m_H(\pt),
\end{multline*}
  where the first
  map is already levelwise an isomorphisms and in particular an
  isomorphism of pro-modules. In order to show that the second map is
  an isomorphism of pro-modules, it remains to show that
  $\II(H)^k \subseteq J_f$ for an appropriate integer $k \ge 1$.
  Equivalently, we want to show that the ideal $\II(H)/J_f$ of
  the quotient ring $\calh^0_H(\pt)/J_f$ is nilpotent. Since
  $\calh^0_H(\pt)$ is Noetherian by
  conditions~\eqref{the:strategy_of_proof_of_the_completion_theorem:Notherian}
  and~\eqref{the:strategy_of_proof_of_the_completion_theorem:fin._gen.}, the ideal
  $\II(H)/J_f$ is finitely generated. Hence it suffices to show
  that $\II(H)/J_f$ is contained in the nilradical, i.e.,
  the ideal consisting of all nilpotent elements, of
  $\calh^0_H(\pt)/J_f$.  The nilradical agrees with the intersection of all the
  prime ideals of $\calh^0_H(\pt)/J_f$
  by~\cite[Proposition~1.8]{Atiyah-McDonald(1969)}.  The preimage of a prime ideal in
  $\calh^0_H(\pt)/J_f$ under the projection $\calh^0_H(\pt) \to \calh^0_H(\pt)/J_f$
  is again a prime ideal.  Hence it remains to show
  that any prime ideal of $\calh^0_H(\pt)$ which contains $I_f$ also
  contains $\II(H)$. But this is guaranteed by
  condition~\eqref{the:strategy_of_proof_of_the_completion_theorem:ideals}. This
  finishes the proof in the case $X = G/H$.

  The general case of a $G$-map $f \colon X \to L$ from a finite
  $G$-$CW$-complex $X$ to a $G$-$CW$-complex $L$
  is done by induction over the dimension $r$ of $X$ and
  subinduction over the number of top-dimensional equivariant cells.
  For the induction step we write $X$ as a $G$-pushout
  \[
  \xycomsquareminus{G/H \times S^{r-1}}{q}{Y}{j}{k}{G/H \times D^r}{Q}{X}
  \]
  In the sequel we equip $G/H \times S^{r-1}$, $Y$ and $G/H \times
  D^r$ with the maps to $L$ given by the composite of $f \colon X \to
  L$ with $k \circ q$, $k$ and $Q$. The long exact Mayer-Vietoris
  sequence of the $G$-pushout above is a long exact sequence of
  $\calh^0_G(L)$-modules and looks like
\begin{multline*}
  \ldots \to \calh^{m-1}(G/H \times   D^r) \oplus \calh^{m+1}_G(Y) \to
  \calh^{m-1}_G(G/H \times S^{r-1}) 
  \to \calh^m_G(X)
  \\
  \to \calh^m_G(G/H \times   D^r) \oplus \calh^m_G(Y) \to
   \calh^m_G(G/H \times S^{r-1})
  \to \ldots.
\end{multline*}
Condition~\eqref{the:strategy_of_proof_of_the_completion_theorem:fin._gen.}
implies that $\calh^m_G(G/H)$ and $\calh^m_G(G/H \times D^r )$
are finitely generated as $\calh^0(\pt)$-modules.  Since $\calh^0(\pt)$
is Noetherian by condition~\eqref{the:strategy_of_proof_of_the_completion_theorem:Notherian},
the $\calh^0(\pt)$-module
$\calh^m_G(X)$ is finitely generated provided that the
$\calh^0(\pt)$-module $\calh^m_G(Y)$ is finitely generated. Thus we
can show inductively that the $\calh^0(\pt)$-module $\calh^m_G(X)$ is
finitely generated for every $m \in \IZ$.  In particular the ring
$\calh^0_G(X)$ is Noetherian. Let $J \subseteq \calh^0_G(X)$ be the
ideal generated by the image of $\II_G(L)$ under the ring
homomorphism $\calh^0_G(L) \to \calh^0_G(X)$. Then for every
$\calh^0_G(X)$-module the obvious map $\{M/\II_G(L)^n \cdot M\}_{n\ge1}
\to \{M/J^n \cdot M\}_{n \ge 1}$ is levelwise an isomorphism and in
particular an isomorphism of $\calh^0_G(X)$-modules. We conclude from
Lemma~\ref{lem:pro-exactness_and_exactness} that the following
sequence of pro-$\calh^0(\pt)$-modules is exact, where $M/I$ stands
for $M/I \cdot M$.
\begin{multline}
  \ldots \to   \{\calh^{m-1}(G/H \times D^r)/\II_G(L)\}_{n\ge 1} \oplus
  \{\calh^{m-1}_G(Y)/\II_G(L)\}_{n\ge 1} 
  \\
  \to \{\calh^{m-1}_G(G/H \times S^{r-1})/\II_G(L)\}_{n\ge 1}
  \\
  \to \{\calh^m_G(X)/\II_G(L)\}_{n\ge 1} 
  \to  \{\calh^m_G(G/H \times D^r)/\II_G(L)\}_{n\ge 1} \oplus
  \{\calh^m_G(Y)/\II_G(L)\}_{n\ge 1}
  \\
  \to \{\calh^m_G(G/H \times S^{r-1})/\II_G(L)\}_{n\ge 1}
  \to \ldots
\label{MV-sequence_I-adic_for_calh_G}
\end{multline}

Applying $EG_{(n-1)} \times_G -$ to the $G$-pushout above yields a
pushout and thus a long exact Mayer-Vietoris sequence
\begin{multline*}
  \ldots \to \calh^{m-1}\left(EG_{(n-1)} \times_G \left(G/H \times
      D^r\right)\right) \oplus \calh^{m-1}_G\left(EG_{(n-1)} \times_G     Y\right)
  \\
   \to \calh^{m-1}_G\left(EG_{(n-1)} \times_G \left(G/H \times S^{r-1}\right)\right)
  \\
  \to  \calh^m_G\left(EG_{(n-1)} \times_G X\right)
  \\
  \to  \calh^m_G\left(EG_{(n-1)} \times_G \left(G/H \times
       D^r\right)\right) \oplus \calh^m_G\left(EG_{(n-1)} \times_G Y\right)
  \\
  \to \calh^m_G\left(EG_{(n-1)} \times_G \left(G/H \times S^{r-1}\right)\right)
  \to \ldots
\end{multline*}

The obvious map 
\[
\left\{\calh^m_G\left(EG_{(n-1)} \times_G Z\right)\right\}_{n \ge 1} \xrightarrow{\cong}
\bigl\{\calh^m_G\bigl(\left(EG\times_G Z\right)_{(n-1)}\bigr)\bigr\}_{n \ge 1}
\]
is an isomorphism of pro-$\calh^0(\pt)$-modules for any finite dimensional
$G$-$CW$-complex $Z$. Hence we obtain a long exact sequence of
pro-$\calh^0(\pt)$-modules
\begin{multline}
  \ldots  \to \bigl\{\calh^{m-1}_G\bigl(\bigl(EG \times_G \bigl(G/H \times
          D^r\bigr)\bigr)_{(n-1)} \bigr)\bigr\}_{n \ge 1} \oplus
  \bigl\{\calh^m_G\bigl(\bigl(EG \times_G
        Y\bigr)_{(n-1)}\bigr)\bigr\}_{n \ge 1}
  \\
  \to \bigl\{\calh^{m-1}_G\bigl(\bigl(EG \times_G \bigl(G/H \times
          S^{r-1}\bigr)\bigr)_{(n-1)}\bigr)\bigr\}_{n \ge 1}
  \\
  \to   \bigl\{\calh^m_G\bigl(\bigl(EG \times_G
        X\bigr)_{(n-1)}\bigr)\bigr\}_{n \ge 1} 
  \\
  \to \bigl\{\calh^m_G\bigl(\bigl(EG \times_G \bigl(G/H \times
          D^r\bigr)\bigr)_{(n-1)} \bigr)\bigr\}_{n \ge 1} \oplus
  \bigl\{\calh^m_G\bigl(\bigl(EG \times_G
        Y\bigr)_{(n-1)}\bigr)\bigr\}_{n \ge 1}
  \\
  \to \bigl\{\calh^m_G\bigl(\bigl(EG \times_G \bigl(G/H \times
          S^{r-1}\bigr)\bigr)_{(n-1)}\bigr)\bigr\}_{n \ge 1}
   \to \ldots
\label{MV-sequence_for_calh((EG_times_X)_(n-1))}
\end{multline}

Now the various maps $\lambda^m_G$ induce a map from the long exact
sequence of pro-$\calh^0(\pt)$-modules~\eqref{MV-sequence_I-adic_for_calh_G} to the long exact sequence
of pro-$\calh^0(\pt)$-modules~\eqref{MV-sequence_for_calh((EG_times_X)_(n-1))}.
The maps for $G/H \times S^{r-1}$, $G/H \times D^r$ and $Y$ are isomorphisms of
pro-$\calh^0(\pt)$-modules by induction hypothesis and by $G$-homotopy
invariance applied to the $G$-homotopy equivalence $G/H \times D^r \to G/H$.
By the Five-Lemma for maps of pro-modules the map
\begin{eqnarray*}
 \lambda^m_G(X) \colon \{\calh^m_G(X)/\II_G(L)^n \cdot \calh^m_G(X)\}_{n \ge 1} & \to &
\{\calh^m_G\left((EG \times_G X)_{(n-1)}\right)\}_{n \ge 1}
\end{eqnarray*}
is an isomorphism of pro-$\calh^0(\pt)$-modules. This finishes the
proof of Theorem~\ref{the:strategy_of_proof_of_the_completion_theorem}.
\end{proof}

The next lemma will be needed to check  
condition~\eqref{the:strategy_of_proof_of_the_completion_theorem:ideals} appearing in
Theorem~\ref{the:strategy_of_proof_of_the_completion_theorem}.

Given an G-cohomology theory $\calh^*_G$.
there is an equivariant version of the Atiyah-Hirzebruch spectral sequence
of $\calh^0(\pt)$-modules
which converges to $\calh^{p+q}(L)$ in the usual sense
provided that $L$  is finite dimensional, and whose $E_2$-term is
\[
E^{p,q}_2 :=  H^p_G(L;\calh^q_G(G/?)),
\]
where $H^p_G(X;\calh^q_G(G/?))$ is the \emph{Bredon cohomology} of
$L$ with coefficients in the $\IZ\Or(G)$-module
sending $G/H$ to $\calh^q_G(G/H)$. If $\calh^*_G$ comes with a
multiplicative structure, then this spectral sequence comes with a multiplicative structure.

\begin{lemma}\label{lem:edge_argument}
Suppose that $L$ is a $l$-dimensional proper $G$-$CW$-complex
for some positive integer $l$. Suppose that for $r = 2,3, \ldots, l$
the differential appearing  in the Atiyah-Hirzebruch spectral sequence
for $L$ and $\calh^*_G$
\[
d^{0,0}_r \colon E^{0,0}_r \to E^{r,1-r}_r
\]
vanishes rationally.

\begin{enumerate}
\item\label{lem:edge_argument:xk_in_the_image_of_edge}
Then we can find  for a given $x \in H^0_G(L;\calh^0_G(G/?))$
a positive integer $k$ such that $x^k$ is contained in the image of the
edge homomorphism
\[
\edge^{0,0} \colon \calh^0_G(L) \to H^0_G(L;\calh^0_G(G/?));
\]

\item\label{lem:edge_argument:improving_the_conditions}
  Let $H \in \calf(L)$, let $\calp \subseteq
  \calh^0_H(\pt)$ be a prime ideal and let $f \colon G/H \to L$ be any
  $G$-map. Suppose that the augmentation ideal
  \[
\II(H) = \ker \left(\calh^0_H(\pt) \xrightarrow{ \calh^0_H(\pr)} \calh^0_H(H)
    \xrightarrow{\ind_{\{1\}\to H}} \calh^0(\pt)\right)
\]
  is contained in $\calp$
  if $\calp$ contains the image of the structure map for $H$ of the inverse limit
  over the orbit category $\Or(G;\calf(L))$  associated to the family $\calf(L)$
  \[
\phi_H \colon \invlim{G/K \in \Or(G;\calf(L))}{\II(K)}  \to \II(H).
\]
  Then condition~\eqref{the:strategy_of_proof_of_the_completion_theorem:ideals}
  appearing in Theorem~\ref{the:strategy_of_proof_of_the_completion_theorem}
  is satisfied for $H$, $\calp$ and $f$.
\end{enumerate}
\end{lemma}
\begin{proof}~\eqref{lem:edge_argument:xk_in_the_image_of_edge}
Consider $x \in H^0_G(L;\calh^0_G(G/?))$. We construct inductively
positive integers $k_1$, $k_2$, $\ldots,$ $k_{l}$ such that
\[
\begin{array}{rclr}
x^{\prod_{i=1}^{r} k_i} & \in & E^{0,0}_{r+1} & \quad \text{ for } r = 1,2, \ldots ,l;
\end{array}
\]
Put $k_1 = 1$. We have $H^0_G(L;\calh^0_G(G/?)) = E^{0,0}_2$ and hence
$x = x^1 = x^{\prod_{i=1}^{1} k_i} \in E^{0,0}_2$.  This finishes
the induction beginning $r = 1$.

In the induction step from $(r - 1)$ to $r \ge 2$
we can assume that we have already constructed
$k_1, \ldots, k_{r-1}$ and shown that
$x^{\prod_{i=1}^{r-1} k_i}$ belongs to $E^{0,0}_r$. Now choose $k_r$ with
$k_r \cdot d^{0,0}_r\left(x^{\prod_{i=1}^{r-1} k_i}\right) = 0$.
This is possible since by assumption
$d^{0,0}_r \otimes_{\IZ} \id_{\IQ} = 0$. For any element
$y \in E^{0,0}_{r}$ one checks inductively for $j = 1,2, \ldots$
\[
d^{0,0}_r(y^j) = j \cdot d^{0,0}_r(y) \cdot y^{j-1}.
\]
This implies
\[
d^{0,0}_r\left(x^{\prod_{i=1}^{r} k_i}\right)
=
d^{0,0}_r\left(\left(x^{\prod_{i=1}^{r-1} k_i}\right)^{k_{r}}\right)
=
k_{r} \cdot d^{0,0}_ {r}\left(x^{\prod_{i=1}^{r-1} k_i}\right) \cdot
\left(x^{\prod_{i=1}^{r-1}}\right)^{k_r -1}
=
0.
\]
Since $E^{0,0}_{r+1}$ is the kernel of
$d^{0,0}_r \colon E^{0,0}_r \to E^{0,0}_{r+1}$, we conclude
$x^{\prod_{i=1}^r k_i} \in E^{0,0}_{r+1}$. Since $L$ is $l$-dimensional, we get
for $k = \prod_{i=1}^{l} k_i$ that $x^k \in E_{\infty}^{0,0}$.
Since $E_{\infty}^{0,0}$ is the image of the edge homomorphism
$\edge^{0,0}$, assertion~\ref{lem:edge_argument:xk_in_the_image_of_edge}
follows.
\\[2mm]~\eqref{lem:edge_argument:improving_the_conditions}
Consider the following commutative diagram
\[
\xymatrix{
&
H^0_G\left(\EGF{G}{\calf(L)};\calh^0_G(G/?)\right)
\ar[d]^-{H^0_G(u)} \ar[r]^-{\alpha}_-{\cong}
&
\invlim{G/K \in \Or(G;\calf(L))}{\calh_K^0(\pt)}
\ar[dddd]^{\Phi_H}
\\
\calh^0_G(L) \ar[r]^-{\edge^{0,0}} \ar[d]^-{\calh^0_G(f)} &
H^0_G\left(L;\calh^0_G(G/?)\right)
\ar[d]^-{H^0_G(f)}
&
\\
\calh^0_G(G/H) \ar[r]^-{\edge^{0,0}}_-{\cong}
\ar[rrdd]^{\ind_{H \to G}}_{\cong}&
H^0_G\left(G/H;\calh^0_G(G/?)\right)
\ar[rdd]^{\quad\ind_{H \to G} \circ H^0_G(i_H)}
&
\\ & & 
\\ & & \calh_H^0(\pt)
}
\]
Here $\alpha$ is the  isomorphism, which sends
$v \in H^0_G\left(\EGF{G}{\calf(L)};\calh^0_G(G/?)\right)$
to the system of elements that  is for $G/K \in \Or(G;\calf(L))$ the image of $v$ under the
homomorphism
\begin{multline*}
H^0_G\left(\EGF{G}{\calf(L)};\calh^0_G(G/?)\right)
\xrightarrow{H^0_G(i_K)}
H^0_G\left(G/K;\calh^0_G(G/?)\right)
\\
\xrightarrow{(\edge^{0,0})^{-1}} \calh^0_G(G/K)
\xrightarrow{\ind_{\{1\} \to K}}
\calh_K^0(\pt),
\end{multline*}
for the up to $G$-homotopy unique $G$-map $i_K \colon G/K \to \EGF{G}{\calf(L)}$. The $G$-map
$u \colon L \to \EGF{G}{\calf(L)}$ is the up to $G$-homotopy
unique $G$-map from $L$ to the classifying space of the family
$\calf(L)$, and $\Phi_H$ is the structure map of the inverse limit for
$H$. We have to prove
that $\II(H)$ is contained in the prime ideal $\calp$ provided that 
$\calp$ contains  the image of $\II_G(L)$ under the composite
$\calh^0_G(L) \xrightarrow{\calh^0_G(f)} \calh^0_G(G/H)
\xrightarrow{\ind_{H \to G}} \calh^0_H(\pt)$.

Consider $a \in \invlim{G/K \in \Or(G;\calf(L))}{\II(K)}$.
Let $x \in H^0_G\left(L;\calh^0_G(G/?)\right)$ be the image of
$a$ under the composite
\begin{multline*}
\invlim{G/K \in \Or(G;\calf(L))}{\II(K)} \to
\invlim{G/K \in \Or(G;\calf(L))}{\calh_K^0(\pt)}
\\
\xrightarrow{\alpha^{-1}}
H^0_G\left(\EGF{G}{\calf(L)};\calh^0_G(G/?)\right)
\\
\xrightarrow{H^0_G\left(u;\calh^0_G(G/?)\right) }
H^0_G\left(L;\calh^0_G(G/?)\right).
\end{multline*}
We conclude from assertion~\eqref{lem:edge_argument:xk_in_the_image_of_edge} that for some
positive number $k$ there is an element $y \in \calh^0_G(L)$ with $\edge^{0,0}(y) = x^k$.
One easily checks that $y$ belongs to $\II_G(L)$, just inspect the diagram above for $H =
\{1\}$. Hence  the composite
\[
\calh^0_G(L) \xrightarrow{\calh^0_G(f)} \calh^0_G(G/H) \xrightarrow{\ind_{H \to G}}
\calh^0_H(\pt)
\]
maps $y$ to $\calp$ by assumption. An easy diagram chase shows that
\[
\phi_H \colon \invlim{G/K \in \Or(G;\calf(L))}{\II(K)} \to \II(H)
\]
maps $a^k$ to $\calp$. Since $\calp$ is a prime ideal and $\phi_H$ is multiplicative,
$\phi_H $ sends $a$ to $\calp$. 
Hence the image of $\phi_H \colon \invlim{G/K \in \Or(G;\calf(L))}{\II(K)} \to \II(H)$ lies  $\calp$.  
Hence we get by assumption $\II(H) \subseteq \calp$.  This finishes the proof of Lemma~\ref{lem:edge_argument}.
\end{proof}


\typeout{--------------------   Section 5 --------------------------------------}

\section{The Segal Conjecture for Infinite Groups}
\label{sec:The_Segal_Conjecture_for_Infinite_Groups}

In this section we prove the Segal
Conjecture~\ref{the:Segal_Conjecture_for_infinite_groups} for infinite groups.  It is just
the Completion Theorem formulated in
Problem~\ref{pro:formulation_of_the_Completion_Theorem} for equivariant stable cohomotopy
$\calh_?^* = \pi_?^*$ under the condition that there is an upper bound on the orders
of finite subgroups on $L$ and $L$ has finite dimension.

\begin{proof}[Proof of Theorem~\ref{the:Segal_Conjecture_for_infinite_groups}]
We want to apply Theorem~\ref{the:strategy_of_proof_of_the_completion_theorem} and
therefore have to prove 
conditions~\eqref{the:strategy_of_proof_of_the_completion_theorem:Notherian}, %
\eqref{the:strategy_of_proof_of_the_completion_theorem:fin._gen.}, %
\eqref{the:strategy_of_proof_of_the_completion_theorem:ideals}
and~\eqref{the:strategy_of_proof_of_the_completion_theorem:finite_groups} appearing there.

Condition~\eqref{the:strategy_of_proof_of_the_completion_theorem:Notherian} appearing there
is satisfied because of $\pi^0_s(\pt) = \IZ$.

Condition~\eqref{the:strategy_of_proof_of_the_completion_theorem:fin._gen.}  is satisfied
because of Example~\ref{exa:equivariant_cohomotopy}.

Next we prove condition~\eqref{the:strategy_of_proof_of_the_completion_theorem:ideals}.
Recall the assumption that there is an upper bound on the orders of finite subgroups of
$L$ and that $L$ is finite dimensional.  Recall that $\calf(L)$ denotes the family of
finite subgroups $H \subseteq G$ with $L^H \not= \emptyset$. We can find by
Example~\ref{exa:equivariant_cohomotopy} for every $q \in \IZ$  with $q \not= 0$ 
a positive integer $C(q)$ such that the order of
$\pi^q_H(\pt)$ divides $C(q)$ for every $H \in \calf(L)$. Furthermore recall that $L$ is
finite dimensional.  Consider the equivariant cohomological Atiyah-Hirzebruch spectral
sequence converging to $\pi^{p+q}_G(L)$. Its $E_2$-term is given by
\[
E^{p,q}_2 = H^p_G(L;\pi^q_G(\pt)).
\]
Therefore $E^{r,1-r}_r$ is
annihilated by multiplication with $C(1-r)$ and hence rationally trivial for $r \ge 2$.
Hence for $r \ge 2$ the differential
\[
d^{0,0}_r \colon E^{0,0}_r \to E^{r,1-r}_r
\]
vanishes rationally. We have shown that the
conditions appearing in Lemma~\ref{lem:edge_argument} are satisfied. Hence in order to
verify  condition~\eqref{the:strategy_of_proof_of_the_completion_theorem:ideals}, it suffices to
prove for any family $\calf$ of subgroups of $G$ with the property that  there exists an upper bound on the
orders of subgroups appearing $\calf$, any $H \in \calf$ and any prime ideal $\calp$ of
the Burnside ring $A(H)$ that $\calp$ contains the augmentation ideal $\II_H$ provided
$\calp$ contains the image of the structure map for $H$ of the inverse limit
\[
\phi_H \colon \invlim{G/K \in \Or(G;\calf)}{\II_K} \to \II_H.
\]

  Fix a finite group $H$. We begin with recalling some basics about
  the prime ideals in the Burnside ring $A(H)$ taken
  from~\cite{Dress(1969)}. In the
  sequel $p$ is a prime number or $p = 0$. For a subgroup $K \subseteq
  H$ let $\calp(K,p)$ be the preimage of $p\cdot \IZ$ under the
  character map for $K$
  \[
\character^H_K \colon A(H) \to \IZ, \quad [S] \mapsto
  |S^K|.
\]
  This is a prime ideal and each prime ideal of $A(H)$ is of
  the form $\calp(K,p)$.  If $\calp(K,p) = \calp(L,q)$, then $p = q$.
  If $p$ is a prime, then $\calp(K,p) = \calp(L,p)$ if and only if
  $(K[p]) = (L[p])$, where $K[p]$ is the minimal normal subgroup of
  $K$ with a $p$-group as quotient. Notice for the sequel that $K[p] =
  \{1\}$ if and only if $K$ is a $p$-group.  If $p = 0$, then
  $\calp(K,p) = \calp(L,p)$ if and only if $(K) = (L)$.

  Fix a prime ideal $\calp = \calp(K,p)$.
  Choose a positive integer $m$ such
  that $|H|$ divides $m$ for all $H\in \calf$. Fix $H \in \calf$.
  Choose a free $H$-set $S$ together with a
  bijection $u \colon S \xrightarrow{\cong} [m]$, where $[m] = \{1,2,
  \ldots , m\}$. Such $S$ exists since $|H|$ divides $m$ and we can
  take for $S$ the disjoint union of $\frac{m}{|H|}$ copies of $H$.
  Thus we obtain an injective group homomorphism
  \[
\rho_u \colon H \to S_m, \quad h \mapsto u \circ l_h \circ u^{-1},
\]
  where $l_h \colon S \to S$ is given by left multiplication with $h$
  and $S_m = \aut([m])$ is the group of permutations of $[m]$. Let
  $S_m[\rho_u]$ denote the $H$-set obtained from $S_m$ by the
  $H$-action $h \cdot \sigma := \rho_u(h) \circ \sigma$.  Let
  $\Syl_p(S_m)$ be the $p$-Sylow subgroup of $S_m$.  Let
  $S_m/\Syl_p(S_m)[\rho_u]$ denote the $H$-set obtained from the
  homogeneous space $S_m/\Syl_p(S_m)$ by the $H$-action given by $h
  \cdot \overline{\sigma} = \overline{\rho_u(h) \circ \sigma}$.  The
  $H$-action on $S_m[\rho_u]$ is free. If for $K \subseteq H$ we have
  $\left(S_m/\Syl_p(S_m)[\rho_u]\right)^K \not= \emptyset$, then for
  some $\sigma \in S_m$ we get $\rho_u(K) \subseteq \sigma \cdot
  \Syl_p(S_m) \cdot \sigma^{-1}$ and hence $K$ must be a $p$-group.

  Suppose that $T$ is another free $H$-set together with a bijection
  $v \colon T \xrightarrow{\cong} [m]$.  Then we can choose an
  $H$-isomorphism $w \colon S \to T$. Let $\tau \in S_m$ be given by
  the composition $v \circ w \circ u^{-1}$. Then $c(\tau) \circ \rho_u
  = \rho_v$ holds, where $c(\tau) \colon S_m \to S_m$ sends $\sigma$
  to $\tau \circ \sigma \circ \tau^{-1}$.  Moreover, left
  multiplication with $\tau$ induces isomorphisms of $H$-sets
\begin{eqnarray*}
S_m[\rho_u] & \cong_H & S_m[\rho_v];
\\
S_m/\Syl_p(S_m)[\rho_u] & \cong_H & S_m/\Syl_p(S_m)[\rho_v].
\end{eqnarray*}
Hence we obtain elements in A(H)
\begin{eqnarray*}
[S_m] & :=  & [S_m[\rho_u]];
\\
{[S_m/\Syl_p(S_m)]} & := & {[S_m/\Syl_p(S_m)[\rho_u]]},
\end{eqnarray*}
which are independent of the choice of $S$ and $u \colon S
\xrightarrow{\cong} [m]$. If $i \colon H_0 \to H_1$ is an
injective group homomorphisms between elements in $\calf$, then
one easily checks that the restriction homomorphism $A(i) \colon
A(H_1) \to A(H_0)$ sends $[S_m]$ to $[S_m]$ and
$[S_m/\Syl_p(S_m)]$ to $[S_m/\Syl_p(S_m)]$. Thus we obtain
elements
\[
[[S_m]], [[S_m/\Syl_p(S_m)]] \in \invlim{G/K \in
  \Or(G;\calf)}{A(K)}
\]
Define elements
\[
|S_m| \cdot 1, |S_m/\Syl_p(S_m)| \cdot 1 \in \invlim{G/K \in
  \Or(G;\calf)}{A(K)}
\]
by the collection of elements $|S_m| \cdot
[K/K]$ and  $|S_m/\Syl_p(S_m)| \cdot [K/K]$ in $A(K)$ for $K \in \calf$.
Thus we get elements
\[
[[S_m]] - |S_m| \cdot 1, [[S_m/\Syl_p(S_m)]] - |S_m/\Syl_p(S_m)| \cdot 1 \in
\invlim{G/K \in \Or(G;\calf)}{\II_K}.
\]
The image of $[[S_m]] - |S_m| \cdot 1$ and
$[[S_m/\Syl_p(S_m)]] - |S_m/\Syl_p(S_m)| \cdot 1$ respectively under the structure map of
the inverse limit $\invlim{G/K \in \Or(G;\calf)}{\II_K}$ for the object $G/H \in
\Or(G;\calf)$ is $[S_m] - |S_m| \cdot [H/H]$ and $[S_m/\Syl_p(S_m)] - |S_m/\Syl_p(S_m)|
\cdot [H/H]$. Hence by assumption 
\begin{eqnarray*}
[S_m] - |S_m| \cdot [H/H] & \in & \calp(K,p);
\\ 
{[S_m/\Syl_p(S_m)]} - |S_m/\Syl_p(S_m)| \cdot [H/H] & \in & \calp(K,p).
\end{eqnarray*}
Therefore $\character^H_K \colon
A(H) \to \IZ$ sends both $[S_m] - |S_m| \cdot [H/H]$ and $[S_m/\Syl_p(S_m)] -
|S_m/\Syl_p(S_m)| \cdot [H/H]$ to elements in $p \IZ$. Since $\character_K^H([S_m] -
|S_m| \cdot [H/H]) = 0 - |S_m|$ for $K \not= \{1\}$, we conclude that $K = \{1\}$ or that
$p \not= 0$.  If $K = \{1\}$, then $\II(H) = \calp(\{1\},0)$ is contained in
$\calp(K,p)$. Suppose that $K \not= \{1\}$. Then $p$ is a prime.  We have
\begin{multline*}
  \character^H_K\left([S_m/\Syl_p(S_m)]  - |S_m/\Syl_p(S_m)| \cdot [H/H]\right)  \\
  = \left|\left(S_m/\Syl_p(S_m)\right)^K\right| - |S_m/\Syl_p(S_m)|.
\end{multline*}
Since this integer must belong to $p \IZ$ and $|S_m/\Syl_p(S_m)|$ is prime to $p$, we get
$\left(S_m/\Syl_p(S_m)\right)^K \not= \emptyset$.  Hence $K$ must be a $p$-group. This
implies $\calp(K,p) = \calp(\{1\},p)$ and therefore $\II(H) = \calp(\{1\},0)
\subseteq \calp(K,p)$.  This finishes the proof of
condition~\eqref{the:strategy_of_proof_of_the_completion_theorem:ideals}.

Condition~\eqref{the:strategy_of_proof_of_the_completion_theorem:finite_groups} follows from
the proof of the Segal Conjecture for a finite group $H$ due to
Carlsson~\cite{Carlsson(1984)}.  This finishes the proof of
Theorem~\ref{the:Segal_Conjecture_for_infinite_groups}.
\end{proof}


\typeout{--------------------   Section 6 --------------------------------------}

\section{An improved Strategy for a Proof of a Completion Theorem}
\label{sec:An_improved_Strategy_for_a_Proof_of_a_Completion_Theorem}

The next result follows from
Theorem~\ref{the:strategy_of_proof_of_the_completion_theorem},
Lemma~\ref{lem:edge_argument} and a construction of a modified
Chern character analogous to the one in~\cite[Theorem~4.6 and Lemma~6.2]{Lueck(2005c)}
which will ensure  that the condition about the differentials in the equivariant Atiyah-Hirzebruch spectral sequence
appearing in Lemma~\ref{lem:edge_argument} is satisfied. We do not give  more details here, since
the interesting case of the Segal Conjecture and of the Atiyah-Segal Completion Theorem are already covered
by Theorem~\ref{the:Segal_Conjecture_for_infinite_groups}  and by~\cite{Lueck-Oliver(2001a)}.

Let $G$ be a (discrete) group. Let $\calf$ be a family of subgroups of $G$ such that there
is an upper bound on the orders of the subgroups appearing $\calf$.  Let $\calh^?_*$ be an
equivariant cohomology theory with values in $R$-modules which satisfies the disjoint
union axiom.  Define a contravariant functor
\begin{equation}
\calh^q_{?}(\pt) \colon \FGINJ \to R\text{-}\MODULES
\label{functor_coming_from_calh}
\end{equation}
with the category $\FGINJ$ of finite groups with injective group homomorphism as  source by sending an injective  homomorphism
$\alpha \colon H \to K$ to the composite
\[
\calh_K^q(\pt) \xrightarrow{\calh^q(\pr)} \calh^q_K(K/H) \xrightarrow{\ind_{\alpha}}
\calh^q_H(\pt),
\]
where $\pr \colon H/K = \ind_{\alpha}(\pt) \to \pt$ is the projection
and $\ind_{\alpha}$ comes from the induction structure of $\calh_?^*$.  Assume that $\calh^?_*$ comes with a multiplicative structure.

\begin{theorem}[Improved Strategy for the proof of
Theorem~\ref{the:Segal_Conjecture_for_infinite_groups}]
\label{the:improved_strategy_of_proof_of_the_completion_theorem}
Suppose that the following
conditions are satisfied.
\begin{enumerate}

\item\label{the:improved_strategy_of_proof_of_the_completion_theorem:Notherian} The ring
  $\calh^0(\pt)$ is Noetherian;

\item\label{the:improved_strategy_of_proof_of_the_completion_theorem:fin._gen.}

  Let $H \subseteq G$ be any finite subgroup and $m \in \IZ$ be any integer.  Then the
  $\calh^0(\pt)$-module $\calh^m_H(\pt)$ is finitely generated, there exists an integer
  $C(H,m)$ such that multiplication with $C(H,m)$ annihilates the torsion-submodule
  $\tors_{\IZ }(\calh^m_H(\pt))$ of the abelian group $\calh^m_H(\pt)$ and the $R$-module 
  $\calh^m_H(\pt)/\tors_{\IZ}(\calh^m_H(\pt))$ is projective;

\item\label{the:improved_strategy_of_proof_of_the_completion_theorem:ideals}

  Let $H$ be any element of $\calf$. Let $\calp \subseteq \calh^0_H(\pt)$ be any prime
  ideal.  Then the augmentation ideal
  \[
\II(H) = \ker \left(\calh^0_H(\pt) \to \calh^0_H(H) \xrightarrow{\cong}
    \calh^0(\pt)\right)
\]
  is contained in $\calp$ if $\calp$ contains the image of the
  structure map for $H$ of the inverse limit
  \[
\phi_H \colon \invlim{G/K \in \Or(G;\calf)}{\II(K)} \to \II(H);
\]

\item\label{the:improved_strategy_of_proof_of_the_completion_theorem:finite_groups}
  The Completion Theorems is true for every finite group
  $H$ in the case $X = L = \pt$ and $f = \id \colon \pt \to \pt$, i.e., for every finite
  group $H$ the map of pro-$\calh^0(\pt)$-modules
\begin{eqnarray*}
 \lambda^m_H(\pt) \colon \{\calh^m_H(\pt)/\II(H)^n\}_{n \ge 1} & \to &
\{\calh^m\left((BH)_{(n-1)}\right)\}_{n \ge 1}
\end{eqnarray*}
defined in~\eqref{map_lambda_of_pro-modules} is an isomorphism of
pro-$\calh^0(\pt)$-modules;

\item\label{the:improved_strategy_of_proof_of_the_completion_theorem:Mackey}
The covariant functor~\eqref{functor_coming_from_calh} extends to a Mackey functor.

\end{enumerate}

Then the Completion Theorem is true for $\calh^?_*$ and every $G$-map
$f \colon X \to L$ from a finite proper $G$-$CW$-complex $X$ to  a proper finite dimensional $G$-$CW$-complex $L$ 
with the property that there is an upper bound on the order of its isotropy groups. 
$L$, i.e.,  the map of
pro-$\calh^0(\pt)$-modules
\begin{eqnarray*}
 \lambda^m_G(X) \colon \left\{\calh^m_G(X)/\II_G(L)^n \cdot \calh^m_G(X)\right\}_{n \ge 1} & \to &
\left\{\calh^m_G\left((EG \times_G X)_{(n-1)}\right)\right\}_{n \ge 1}
\end{eqnarray*}
defined in~\eqref{map_lambda_of_pro-modules} is an isomorphism of
pro-$\calh^0(\pt)$-modules.
\end{theorem}

\begin{remark}\label{rem:advantage_of_the_improved_version} 
  The advantage of Theorem~\ref{the:improved_strategy_of_proof_of_the_completion_theorem}
  in comparison with Theorem~\ref{the:strategy_of_proof_of_the_completion_theorem} is that
  the conditions do not involve $L$ and $f \colon X \to L$ anymore and do only depend on
  the functor $\calh^q_?(\pt) \colon \FGINJ \to \IZ\text{ -}\MODULES$.  If one considers
  the case $R = \IZ$ and assumes $\calh^0(\pt) = \IZ$, then
  condition~\eqref{the:improved_strategy_of_proof_of_the_completion_theorem:Notherian} is
  obviously satisfied and
  condition~\eqref{the:improved_strategy_of_proof_of_the_completion_theorem:fin._gen.}
  reduces to the condition that for any finite subgroup $H \subseteq G$ and any integer $m
  \in \IZ$ the abelian group $\calh^m_H(\pt)$ is finitely generated.
\end{remark}

\begin{remark}[Family version]\label{rem:family_version}
We mention without proof that there is a also a family version of
Theorem~\ref{the:Segal_Conjecture_for_infinite_groups}. Its formulation is analogous to the one
of the family version of the Atiyah-Segal Completion Theorem for infinite groups,
see~\cite[Section~6]{Lueck-Oliver(2001b)}.
\end{remark}


\typeout{-------------------- References -------------------------------}



\end{document}